\documentclass[final,leqno,onefignum,onetabnum]{siamltex1213}

\usepackage{amsmath} 
\usepackage{amssymb}

\usepackage{enumerate}

\newtheorem{example}[theorem]{Example}
\newtheorem{remark}[theorem]{Remark}

\newcommand  \esssup {\mathop{\text{ess} \sup} }

\newcommand \eps   {\varepsilon}

\newcommand \R   {\mathbb{R}}
\newcommand \C   {\mathbb{C}}

\newcommand \K   {\mathcal{K}}
\newcommand \Kinf{\mathcal{K_\infty}}

\newcommand \KL  {\mathcal{KL}}

\newcommand \LL  {\mathcal{L}}

\newcommand{\Uc}{\ensuremath{\mathcal{U}}}
\newcommand{\Vc}{\ensuremath{\mathcal{V}}}

\newcommand \Iff   {\Leftrightarrow}

\newcounter{syscounter}
\newenvironment{sysnum}{\begin{list}{($\Sigma{\arabic{syscounter}}$)}%
{\settowidth{\labelwidth}{($\Sigma4$)}
\settowidth{\leftmargin}{($\Sigma4$)~}%
\usecounter{syscounter}}}
{\end{list}}

\title{
 Monotonicity Methods for Input-to-State Stability of Nonlinear Parabolic PDEs with Boundary Disturbances
\thanks{A. Mironchenko was supported by the German Research Foundation (DFG) within the project "Input-to-state stability and stabilization of distributed parameter systems" (grant Wi 1458/13-1)}
} 

\author{Andrii Mironchenko \thanks{Andrii Mironchenko is with Faculty of Computer Science and Mathematics, University of
  Passau, Germany 
(\email{andrii.mironchenko@uni-passau.de}). Corresponding author.}
\and Iasson Karafyllis
\thanks{Iasson Karafyllis Department of Mathematics, National Technical University of Athens, Greece 
(\email{iasonkar@central.ntua.gr}).
}
\and Miroslav Krstic
\thanks{Miroslav Krstic is with Department of Mechanical and Aerospace Engineering, University of California, San Diego, USA 
(\email{krstic@ucsd.edu}).
}
}

\date{\today}

\begin{document}
\maketitle

\begin{abstract}
We introduce a monotonicity-based method for studying input-to-state stability (ISS) of nonlinear parabolic equations with boundary inputs.
We first show that a monotone control system is ISS if and only if it is ISS w.r.t. constant inputs. Then we show by means of classical maximum principles that nonlinear parabolic equations with boundary disturbances are monotone control systems. 

With these two facts, we establish that ISS of the original nonlinear parabolic PDE with constant \textit{boundary disturbances} 
is equivalent to ISS of a closely related nonlinear parabolic PDE with constant \textit{distributed disturbances} and zero boundary condition. The last problem is conceptually much simpler and can be handled by means of various recently developed techniques.
As an application of our results, we show that the PDE backstepping controller which stabilizes linear reaction-diffusion equations from the boundary is robust with respect to additive actuator disturbances. 
\end{abstract}



\pagestyle{myheadings}
\thispagestyle{plain}
\markboth{MONOTONICITY AND ISS OF PARABOLIC SYSTEMS}{MONOTONICITY AND ISS OF PARABOLIC SYSTEMS}

\section{Introduction}

The concept of input-to-state stability (ISS), which unified Lyapunov and input-output approaches, plays a foundational role in nonlinear control theory \cite{Son08}. It is central for robust stabilization of nonlinear systems \cite{FrK08,KKK95}, design of nonlinear observers \cite{ArK01}, analysis of large-scale networks \cite{JTP94,DRW07} etc.
Interest in ISS of infinite-dimensional systems has been steadily growing since the end of 1990-s, but for a decade it was limited to time-delay systems, see e.g. \cite{Tee98, PeJ06, PKJ08, Krs08d}. Recently, a rapid development of an ISS theory of abstract infinite-dimensional systems \cite{JLR08, KaJ11, DaM13b, MiI15b, MiW17b, Mir16} and more specifically of partial differential equations \cite{MaP11, PrM12, MiI15b, AVP16, CDP17, PiO17, TPT16} has been taking place. Important results achieved include characterizations of ISS and local ISS for a rather general class of nonlinear infinite-dimensional systems \cite{MiW17b, Mir16}, abstract nonlinear small-gain theorems \cite{KaJ11, MiI15b}, applications of ISS Lyapunov theory to analysis and control of various classes of PDE systems \cite{PrM12, AVP16, TPT16, CDP17, PiO17} etc.
However, most of these papers are devoted to PDEs with distributed inputs.
\textit{In this work we study input-to-state stability (ISS) of nonlinear parabolic partial differential equations (PDEs) with boundary disturbances on multidimensional spatial domains.} This question naturally arises in such fundamental problems of PDE control as robust boundary stabilization of PDE systems, design of robust boundary observers, stability analysis of cascades of parabolic and ordinary differential equations (ODEs) etc.

It is well known, that PDEs with boundary disturbances can be viewed as evolution equations in Banach spaces with unbounded input (disturbance) operators. This makes the analysis of such systems much more involved than stability analysis of PDEs with distributed disturbances (which are described by bounded input operators), even in the linear case. 

At the same time, ISS of linear parabolic systems w.r.t. boundary disturbances has been studied in several recent papers using different methodologies \cite{AWP12, KaK16b, KaK17a, JNP16}.
In \cite{AWP12} the authors attacked this problem by means of Lyapunov methods. Due to the methodology followed in \cite{AWP12} boundary disturbances must be differentiated w.r.t. time (which is also needed when one tries to switch from boundary to distributed disturbances). As a result, in \cite{AWP12} ISS of a parabolic system w.r.t. $C^1$ norm has been achieved (known as $D^1$-ISS, see \cite[p. 190]{Son08}).

In \cite{KaK16b, KaK17a, KaK17b} linear parabolic PDEs with Sturm-Liouville operators over 1-dimensional spatial domain have been treated by using two different methods: (i) the spectral decomposition of the solution, and (ii) the approximation of the solution by means of a  finite-difference scheme. 
This made possible to avoid differentiation of boundary disturbances, and to obtain ISS of classical solutions w.r.t. $L^\infty$ norm of disturbances, as well as in weighted $L^2$ and $L^1$ norms. 
An advantage of these methods is that this strategy can be applied also to other types of linear evolution PDEs. At the same time, for multidimensional spatial domains, the computations can become quite complicated.

In \cite{JNP16} ISS of linear boundary control systems has been approached by means of methods of semigroup and admissibility theory. In particular, interesting and nontrivial relations between ISS and integral input-to-state stability (iISS) have been obtained. 
An advantage of this method is that it encompasses very broad classes of linear infinite-dimensional systems, but due to the lack of proper generalizations of admissibility to general nonlinear systems, it cannot be applied (at least at current stage) to nonlinear distributed parameter systems.

In this paper, we propose a novel method for investigation of parabolic PDEs with boundary disturbances. In contrast to previous results, we do not restrict ourselves to linear equations over 1-dimensional spatial domains. Our results are valid for a class of nonlinear equations over multidimensional bounded domains with a sufficiently smooth boundary. Our method is based on the concept of monotone control systems introduced in \cite{AnS03} and inspired by the theory of monotone dynamical systems pioneered by M. Hirsch in 1980-s in a series of papers, beginning with \cite{Hir82}. For the introduction to this theory, one may consult \cite{Smi95}.
An early effort to use monotonicity methods to study of ISS of infinite-dimensional systems has been made in \cite{DaM10} to show ISS of a quite special class of parabolic systems with distributed inputs and Neumann boundary conditions.

In this paper, we study ISS of a broad class of nonlinear parabolic equations with boundary disturbances.
Using maximum principles \cite{PrW12, Fri83}, we show under certain regularity assumptions that such systems are monotone control systems. 
Furthermore, we show that a monotone control system (and in particular a nonlinear parabolic PDE system with boundary disturbances) is ISS if and only if it is ISS over a much smaller class of inputs (e.g. constant inputs).
This is achieved by proving that, for any given disturbance, there is a larger constant disturbance which leads to the larger deviation from the origin, and hence, in a certain sense, the constant disturbances are the "worst case" ones.

This nice property helps us to show that ISS of the original nonlinear parabolic PDE with constant \textit{boundary disturbances} is equivalent to ISS of a closely related nonlinear parabolic PDE with constant \textit{distributed disturbances}. 
The latter problem has been studied extensively in the last years, and a number of powerful results are available for this class of systems. 
In particular, in \cite{MaP11}, constructions of strict Lyapunov functions for certain nonlinear parabolic systems have been provided.
In \cite{Mir16}, the Lyapunov characterizations of local ISS property have been shown for a general class of infinite-dimensional systems, which encompass in particular parabolic systems. In \cite{MiI15b, DaM13}, ISS and integral ISS small-gain theorems for nonlinear parabolic systems interconnected via spatial domain have been proved which give powerful tools to study stability of large-scale parabolic systems on the basis of knowledge of stability of its components. 
With the help of the results in this paper, this machinery can be used to analyze ISS of nonlinear parabolic PDEs with boundary inputs. 


Finally, we apply the derived ISS criteria to the problem of robust stabilization of linear parabolic systems by means of a boundary control in the presence of actuator disturbances. We prove that stabilizing controllers, achieved by PDE backstepping method (see for instance \cite{KrS08, SmK10}), are in fact ISS stabilizing controllers w.r.t. actuator disturbances.

Next, we introduce some notation used throughout these notes.
By $\R_+$ we denote the set of nonnegative real numbers. For $z \in \R^n$ the Euclidean norm of $z$ is denoted by $|z|$.
For any open set $G\subset \R^n$ we denote by $\partial G$ the boundary of $G$; by $\overline{G}$ the closure of $G$ and by $\mu(G)$ the Lebesgue measure of $G$.
Also for such $G$ and any $p\in[1,+\infty)$ we denote 
by $L^p(G)$ the space of Lebesgue measurable functions $y$ with $\|y\|_p=\Big(\int_G |y(z)|^p dz\Big)^{1/p}$, 
by $L^\infty(G)$ the space of Lebesgue measurable functions $y$ with $\|y\|_\infty=\esssup_{z\in G}|y(z)|$, 
and by $H^k(G)$ the set of $g\in L^2(G)$, which possess weak derivatives up to the $k$-th order, all of which belong to $L^2(G)$.
$C^k(G)$ consists of $k$ times continuously differentiable functions defined on $G$ and $C_0^1(G)$ consists of a functions from $C^1(G)$ which have a compact support.
$H^{1}_0(G)$ is a closure of $C_0^1(G)$ in the norm of $H^1(G)$. 
If $I \subset \R_+$, then $C^{1,2}(I\times G)$ means the set of functions mapping $I\times G$ to $\R$, which are continuously differentiable w.r.t. the first argument and possess continuous second derivatives w.r.t. the second arguments.

Also we will use the following classes of comparison functions.
\begin{equation*}
\begin{array}{ll}
{\K} &:= \left\{\gamma:\R_+ \to \R_+ \ : \ \gamma\mbox{ is continuous and strictly increasing, }\gamma(0)=0\right\}\\
{\K_{\infty}}&:=\left\{\gamma\in\K \ :\ \gamma\mbox{ is unbounded}\right\}\\
{\LL}&:=\left\{\gamma:\R_+ \to \R_+ \ :\ \gamma\mbox{ is continuous and decreasing with}
 \lim\limits_{t\rightarrow\infty}\gamma(t)=0 \right\}\\
{\KL} &:= \left\{\beta: \R_+^2 \to \R_+ \ : \ \beta(\cdot,t)\in{\K},\ \forall t \geq 0,\  \beta(r,\cdot)\in {\LL},\ \forall r >0\right\}
\end{array}
\end{equation*}


\section{Monotonicity of control systems}
\label{sec:Framework_Monotonicity}


We start with a definition of a control system.
\begin{definition}
\label{Steurungssystem}
Consider a triple $\Sigma=(X,\Uc,\phi)$, consisting of 
\begin{enumerate}[(i)]  
 \item A normed linear space $(X,\|\cdot\|_X)$, called the {state space}, endowed with the norm $\|\cdot\|_X$.
 \item A set of  input values $U$, which is a nonempty subset of a certain normed linear space.
%
    %
    %
%
 \item A normed linear space of inputs $\Uc \subset \{f:\R_+ \to U\}$ endowed with the norm $\|\cdot\|_{\Uc}$.
   We assume that $\Uc$ satisfies \textit{the axiom of shift invariance}, which states that for all $u \in \Uc$ and all $\tau\geq0$ the time
shift $u(\cdot + \tau)$ is in $\Uc$.

\item A family of nonempty sets $\{\Uc(x) \subset \Uc: x \in X\}$, where $\Uc(x)$ is a set of admissible inputs for the state $x$.

\item A transition map $\phi:\R_+ \times X \times \Uc \to X$, defined for any $x\in X$ and any $u\in \Uc(x)$ on a certain subset of $\R_+$.
\end{enumerate}
The triple $\Sigma$ is called a (forward-complete) \textbf{control system}, if the following properties hold:
\begin{sysnum}
 \item \textit{Forward-completeness}: for every $x\in X$, $u\in\Uc(x)$ and for all $t \geq 0$ the value 
$\phi(t,x,u) \in X$ is well-defined.
 \item\label{axiom:Identity} \textit{The identity property}: for every $(x,u) \in X \times \Uc(x)$
          it holds that $\phi(0,x,u)=x$.
 \item \label{axiom:Causality}\textit{Causality}: for every $(t,x,u) \in \R_+ \times X \times
          \Uc(x)$, for every $\tilde{u} \in \Uc(x)$, such that $u(s) =
          \tilde{u}(s)$, $s \in [0,t]$ it holds that $\phi(t,x,u) = \phi(t,x,\tilde{u})$.
  \item \label{axiom:Cocycle} \textit{The cocycle property}: for all $t,h \geq 0$, for all
                  $x \in X$, $u \in \Uc(x)$ we have $u(t+\cdot) \in \Uc(\phi(t,x,u))$ and
\[
\phi(h,\phi(t,x,u),u(t+\cdot))=\phi(t+h,x,u).
\]
\end{sysnum}
\end{definition}
In the above definition $\phi(t,x,u)$ denotes the state of a system at the moment $t \in
\R_+$ corresponding to the initial condition $x \in X$ and the input $u \in \Uc(x)$.
A pair $(x,u)\in X\times\Uc(x)$ is referred to as an \textit{admissible pair}.

\begin{remark}\em
\label{rem:Compatibility_Conditions}
For wide classes of systems, in particular for ordinary differential equations, one can assume that $\Uc(x)=\Uc$ for all $x\in X$, that is every input is admissible for any state. On the other hand, the classical solutions of PDEs with Dirichlet boundary inputs have to satisfy compatibility conditions (see Section~\ref{ISS_nonlinear_parabolic_PDEs}), and hence for such systems $\Uc(x)\neq \Uc$ for any $x\in X$.
Another class of systems for which one cannot expect that $\Uc(x)=\Uc$ for all $x\in X$ are differential-algebraic equations 
(DAEs), see e.g. \cite{KuM06, KrT15}.
\end{remark}

\begin{definition}
A subset $K \subset X$ of a normed linear space $X$ is called a \textit{positive cone} if $K \cap (-K)=\{0\}$ 
and for all $a \in \R_+$ and all $x,y \in K$ it follows that $ax \in K$; $x+y \in K$.
\end{definition}

\begin{definition}
A normed linear space $X$ together with a cone $K \subset X$ is called an \textit{ordered normed linear space} (see \cite{Kra64}), which we denote $(X,K)$ with an order $\leq$ given by $x \leq y\ \Iff \  y-x \in K$. Analogously $x \geq y\ \Iff \  x-y \in K$.
\end{definition}

\begin{definition}
\index{control system!ordered}
We call a control system $\Sigma=(X,\Uc,\phi)$ ordered, if $X$ and $\Uc$ are ordered normed linear spaces.
\end{definition}

An important for applications subclass of control systems are monotone control systems:
\begin{definition} 
An ordered control system $\Sigma=(X,\Uc,\phi)$ is called monotone, provided for all $t \geq 0$, all $x_1, x_2 \in X$ with $x_1 \leq x_2$ and all $u_1 \in\Uc(x_1), u_2 \in \Uc(x_2)$ with $u_1 \leq u_2$ it holds that 
$\phi(t,x_1,u_1)  \leq \phi(t,x_2,u_2)$.
\end{definition}

To treat situations when the monotonicity w.r.t. initial states is not available, the following definition is useful:
\begin{definition} 
An ordered control system $\Sigma=(X,\Uc,\phi)$ is called monotone w.r.t. inputs, provided for all $t \geq 0$, all $x \in X$ and all $u_1,u_2 \in \Uc(x)$ with $u_1 \leq u_2$ it holds that $ \phi(t,x,u_1)  \leq \phi(t,x,u_2)$.
\end{definition}


\section{Input-to-state stability of monotone control systems}
\label{sec:ISS_monotone_control_systems}

Next we introduce the notion of input-to-state stability, which will be central in this paper.
\begin{definition} 
\label{def:ISS_Uc}
Let $\Sigma=\left(X,\mathcal{U},\phi \right)$ be a control system. Let $\Uc_{c}$ be a subset of $\mathcal{U}$. 
System $\Sigma$ is called \textit{input-to-state stable (ISS) with respect to inputs in $\Uc_{c}$} if there exist functions $\beta \in \KL$, $\gamma \in \K$ such that for every $x \in X$ for which $\Uc(x ) \bigcap \Uc_{c}$ is non-empty, the following estimate holds for all $u \in \Uc(x ) \bigcap \Uc_{c}$ and $t\ge 0$:
\begin{eqnarray}
\| \phi(t,x,u) \|_X \leq \beta(\| x \|_X,t) + \gamma( \|u\|_{\Uc}).
\label{eq:ISS_estimate}
\end{eqnarray}
If $\Sigma$ is ISS w.r.t. inputs from $\Uc$, then $\Sigma$ is called \textit{input-to-state stable (ISS)}.
\end{definition}

For applications the following notion, which is stronger than ISS is of importance:
\begin{definition} 
\label{def:exp-ISS_Uc}
Let $\Sigma=\left(X,\mathcal{U},\phi \right)$ be a control system. Let $\Uc_{c}$ be a subset of $\mathcal{U}$. 
System $\Sigma$ is called \textit{exponentially input-to-state stable (exp-ISS) with respect to inputs in $\Uc_{c}$} if there exist constants $M,a$ and $\gamma\in\Kinf$ such that for every $x \in X$ for which $\Uc(x ) \bigcap \Uc_{c}$ is non-empty, the following estimate holds for all $u \in \Uc(x ) \bigcap \Uc_{c}$ and $t\ge 0$:
\begin{eqnarray}
\| \phi(t,x,u) \|_X \leq Me^{-at}\| x \|_X + \gamma(\|u\|_{\Uc}).
\label{eq:exp-ISS_estimate}
\end{eqnarray}
If $\Sigma$ is exp-ISS w.r.t. inputs from $\Uc$, then $\Sigma$ is called \textit{exponentially input-to-state stable}.
If in addition $\gamma$ can be chosen to be linear, then $\Sigma$ is exp-ISS with a linear gain function.
\end{definition}

\begin{remark}\em
\label{rem:eISS}
Exponential ISS as defined in Definition~\ref{def:exp-ISS_Uc} has been used in \cite{DaM13,LiH15} under the name of eISS.
A similar notion was used in the context of stochastic systems in \cite[Definition 2.4]{SpT03}.
In some works (see e.g. \cite{PrW96}) exponential ISS is defined in a different way (under the name of expISS) and is related to the existence of an ISS Lyapunov function with an exponential decay rate along the trajectories, which is not equivalent to the notion, introduced in Definition~\ref{def:exp-ISS_Uc}.
\end{remark}

We are also interested in the stability properties of control systems in absence of inputs.
\begin{definition}
\label{def:0UGAS_Uc}
A control system $\Sigma=\left(X,\mathcal{U},\phi \right)$ is {\it globally asymptotically
stable at zero uniformly with respect to the state} (0-UGAS), if there
exists a $ \beta \in \KL$, such that for all $x \in X$: $0 \in \Uc(x)$ and for all $
t\geq 0$ it holds that
\begin{equation}
\label{UniStabAbschaetzung}
\left\| \phi(t,x,0) \right\|_{X} \leq  \beta(\left\| x \right\|_{X},t) .
\end{equation}
\end{definition}

%
%

It may be hard to verify the ISS estimate \eqref{eq:ISS_estimate} for all admissible pairs of states and inputs.
Therefore a natural question appears: to find a smaller set of admissible states and inputs, so that validity of the ISS estimate \eqref{eq:ISS_estimate} implies ISS of the system for all admissible pairs (possibly with larger $\beta$ and $\gamma$). 
For example, for wide classes of systems it is enough to check ISS estimates on the properly chosen dense subsets of the space of admissible pairs (a so-called "density argument", see e.g. \cite[Lemma 2.2.3]{Mir12}). As Propositions~\ref{prop:ISS_Monotone_Systems}, \ref{prop:ISS_Monotone_states_inputs_Systems} show, for monotone control systems such sets can be much more sparse.

We start with the case, when all $(x,u) \in X \times \Uc$ are admissible pairs.
\begin{proposition}
\label{prop:ISS_Monotone_Systems}
Let $(X,\Uc,\phi)$ be a control system that is monotone with respect to inputs  with $\Uc(x)=\Uc$ for all $x\in X$. Furthermore, let $\Uc_{c}$ be a subset of $\mathcal{U}$ and let the following two conditions hold:
\begin{enumerate}
 \item[(i)] There exists $\rho \in\Kinf$ so that for any $x_-,x,x_+ \in X$ satisfying $x_-\leq x \leq x_+$ it holds that
\begin{eqnarray*}
\|x\|_X \leq \rho(\|x_-\|_X + \|x_+\|_X).
\end{eqnarray*}
 \item[(ii)] There exists $\eta \in \Kinf$ so that for any $u \in \Uc$ there are $u_-,u_+ \in \Uc_{c}$, satisfying $u_-\leq u \leq u_+$ and $\|u_-\|_{\Uc} \leq \eta(\|u\|_{\Uc})$ and $\|u_+\|_{\Uc} \leq \eta(\|u\|_{\Uc})$. 
\end{enumerate}
Then $(X,\Uc,\phi)$ is ISS if and only if $(X,\Uc,\phi)$ is ISS w.r.t. inputs in $\Uc_{c}$.

Moreover, if $\rho$ is linear, then $\Sigma$ is exp-ISS if and only if $\Sigma$ is exp-ISS w.r.t. inputs in $\Uc_{c}$.
If additionally $\Sigma$ is exp-ISS w.r.t. inputs in $\Uc_{c}$ with a linear gain function and $\eta$ is linear, then $\Sigma$ is exp-ISS
with a linear gain function.
\end{proposition}

\begin{proof}
The "$\Rightarrow$" direction is evident, thus we show "$\Leftarrow$".
Let $(X,\Uc,\phi)$ be ISS w.r.t. inputs in $\Uc_{c}$.
Pick any $u\in \Uc$ and let $u_-,u_+ \in \Uc_{c}$ be as in assumption (ii) of the proposition.
Due to monotonicity of $(X,\Uc,\phi)$ with respect to inputs and since $\Uc(x)=\Uc$ for any $x\in X$ we have that for any $t\geq 0$ and any $x \in X$ 
\begin{eqnarray}
\phi(t,x,u_-) \leq \phi(t,x,u) \leq \phi(t,x,u_+).
\label{eq:tmp_estimate}
\end{eqnarray}
In view of the assumption (i) we have that
\begin{eqnarray*}
\|\phi(t,x,u)\|_X \leq \rho\big(\|\phi(t,x,u_-)\|_X + \|\phi(t,x,u_+)\|_X\big),
\end{eqnarray*}
which implies due to ISS of $(X,\Uc,\phi)$ w.r.t. inputs in $\Uc_{c}$, that there exist $\beta \in \KL$ and $\gamma\in\Kinf$ so that 
\begin{eqnarray*}
\|\phi(t,x,u)\|_X &\leq& \rho\Big(\beta\big(\| x \|_X,t\big) + \gamma\big( \|u_-\|_{\Uc}\big)  + \beta\big(\| x \|_X,t\big) + \gamma\big( \|u_+\|_{\Uc}\big)\Big)\\
&\leq& \rho\Big(2\beta\big(\| x \|_X,t\big) + 2 \gamma \circ \eta\big( \|u\|_{\Uc}\big)\Big),
\end{eqnarray*}
which due to the trivial inequality $\rho(a+b)\leq \rho(2a) + \rho(2b)$, valid for all $a,b\in \R_+$, implies 
\begin{eqnarray*}
\|\phi(t,x,u)\|_X &\leq& \rho\Big(4\beta\big(\| x \|_X,t\big)\Big) + \rho\Big(4 \gamma \circ \eta\big( \|u\|_{\Uc}\big)\Big)\\
&=:& \hat\beta\big(\| x \|_X,t\big) + \hat\gamma\big( \|u\|_{\Uc}\big),
\end{eqnarray*}
where $\hat\beta:(r,t) \mapsto \rho\big(4\beta(r,t)\big)$ is a $\KL$ function and $\hat\gamma(r):r \mapsto \rho\big(4 \gamma \circ \eta(r)\big)$ is a $\Kinf$ function.
This shows ISS of $(X,\Uc,\phi)$.

If $\rho$ is linear, and $\Sigma$ is exp-ISS w.r.t. inputs in $\Uc_{c}$, then above argument justifies exp-ISS of $\Sigma$.
Clearly, if $\eta$ and $\gamma$ are linear, then also $\tilde\gamma$ is linear.
\end{proof}

\begin{example}
\label{rem:Results_in_ODE_context}
Consider ordinary differential equations of the form
\begin{eqnarray}
\dot{x}=f(x,u),
\label{eq:ODEsys}
\end{eqnarray}
with $X=\R^n$ with an order induced by the cone $\R^n_+$, $\Uc:=L^\infty(\R_+,\R^m)$ with the order induced by the cone $L^\infty(\R_+,\R^m_+)$ and $\Uc(x)=\Uc$ for all $x\in\R^n$. 
Under assumptions that $f$ is Lipschitz continuous w.r.t. the first argument uniformly w.r.t. the second one and that \eqref{eq:ODEsys} is forward complete,
\eqref{eq:ODEsys} defines a control system $(X,\Uc,\phi)$, where $\phi(t,x_0,u)$ is a state of \eqref{eq:ODEsys} at the time $t$ corresponding to $x(0)=x_0$ and input $u$.

Assume that \eqref{eq:ODEsys} is monotone w.r.t. inputs with such $X$ and $\Uc$ and consider $\Uc_{c}:=\{u\in\Uc:\exists k\in\R^m:\ u(t)=k \text{ for a.e. } t\in\R_+\}$. It is easy to verify that assumptions of Proposition~\ref{prop:ISS_Monotone_Systems} are fulfilled, and hence \eqref{eq:ODEsys}
is ISS iff it is ISS w.r.t. the inputs with constant in time controls.

This may simplify analysis of ISS of monotone ODE systems since input-to-state stable ODE systems with constant inputs have some specific properties,
see e.g. \cite[pp. 205--206]{Son08}.
\end{example}


As we argued in Remark~\ref{rem:Compatibility_Conditions}, for many systems $\Uc(x) \neq \Uc$ for some $x\in X$. For such systems Proposition~\ref{prop:ISS_Monotone_Systems} is inapplicable, due to the fact that existence of inputs $u_-,u_+ \in \Uc_{c}$ for a given initial condition $x$ and for an input $u$, satisfying assumption (ii) of Proposition~\ref{prop:ISS_Monotone_Systems} does not guarantee that the pairs $(x,u_-)$ and $(x,u_+)$ are admissible, which is needed for the proof of Proposition~\ref{prop:ISS_Monotone_Systems}.
However, if $(X,\Uc,\phi)$ is in addition monotone w.r.t. states, the following holds:
\begin{proposition}
\label{prop:ISS_Monotone_states_inputs_Systems}
Let $\Sigma=\left(X,\mathcal{U},\phi \right)$ be a control system that is monotone with respect to states and inputs. 
Assume that $\Uc_{c}$ is a subset of $\mathcal{U}$ and let the following two conditions hold:
\begin{itemize}
   \item[(i)]There exists $\rho \in \Kinf $ so that for any $x_{-} ,x,x_{+} \in X$ satisfying $x_{-} \le x\le x_{+} $ it holds that
\begin{equation} \label{GrindEQ__19_} 
\left\| x\right\| _{X} \le \rho \left(\left\| x_{-} \right\| _{X} +\left\| x_{+} \right\| _{X} \right).
\end{equation} 
\item[(ii)] There exist $\eta ,\xi \in \Kinf $ so that for every $x\in X$, $u\in \Uc(x)$ and for every $\varepsilon >0$ there exist $x_{-} ,x_{+} \in X$ and $u_{-} \in \Uc(x_{-} ) \bigcap \Uc_{c}$, $u_{+} \in \Uc(x_{+}) \bigcap \Uc_{c}$, satisfying $x_{-} \le x\le x_{+} $, $u_{-} \le u\le u_{+} $, so that the estimates
\begin{equation} \label{GrindEQ__20_} 
\max \left(\left\| u_{-} \right\| _{\mathcal{U}} ,\left\| u_{+} \right\| _{\mathcal{U}} \right)\le \eta \left(\left\| u\right\| _{\mathcal{U}} +\varepsilon \right),
\end{equation} 
\begin{equation} \label{GrindEQ__21_} 
\max \left(\left\| x_{-} \right\| _{X} ,\left\| x_{+} \right\| _{X} \right)\le \xi \left(\left\| x\right\| _{X} +\left\| u\right\| _{\mathcal{U}} +\varepsilon \right).
\end{equation}
\end{itemize}
hold. Then $\Sigma$ is ISS if and only if $\Sigma$ is ISS w.r.t. inputs in $\Uc_{c}$.

Moreover, if $\rho$ and $\xi$ are linear, then $\Sigma$ is exp-ISS if and only if $\Sigma$ is exp-ISS w.r.t. inputs in $\Uc_{c}$.
If additionally $\Sigma$ is exp-ISS w.r.t. inputs in $\Uc_{c}$ with a linear gain function and $\eta$ is linear, then $\Sigma$ is exp-ISS with a linear gain function.
\end{proposition}

\begin{proof}
The "$\Rightarrow$" direction is evident, thus we show "$\Leftarrow$".
Let $(X,\Uc,\phi)$ be ISS w.r.t. inputs in $\Uc_{c}$.
Pick any $x \in X$, $u\in \Uc(x)$ and $\eps>0$ and let $x_-,x_+ \in X$ and $u_- \in \Uc(x_-)\bigcap \Uc_{c}, u_+ \in \Uc(x_+)\bigcap \Uc_{c}$ be initial states and constant in time inputs as in assumption (ii).
Due to monotonicity of $(X,\Uc,\phi)$ for any $t\geq 0$ it holds that 
\begin{eqnarray*}
\phi(t,x_-,u_-) \leq \phi(t,x,u) \leq \phi(t,x_+,u_+).
\end{eqnarray*}
In view of assumption (i) we have 
\begin{eqnarray*}
\|\phi(t,x,u)\|_X \leq \rho\big(\|\phi(t,x_-,u_-)\|_X + \|\phi(t,x_+,u_+)\|_X\big),
\end{eqnarray*}
and ISS of $(X,\Uc,\phi)$ w.r.t. inputs in $\Uc_{c}$ implies existence of $\beta \in \KL$ and $\gamma\in\Kinf$ so that 
\begin{eqnarray*}
\|\phi(t,x,u)\|_X &\leq& \rho\Big(\beta\big(\| x_- \|_X,t\big) + \gamma\big( \|u_-\|_{\Uc}\big)  + \beta\big(\| x_+ \|_X,t\big) + \gamma\big( \|u_+\|_{\Uc}\big)\Big)\\
&\leq& \rho\Big(2\beta\big(\xi(\| x \|_X + \|u\|_{\Uc}+\eps),t\big) + 2 \gamma \circ \eta\big( \|u\|_{\Uc} + \eps \big)\Big).
\end{eqnarray*}
Taking the limit $\eps\to +0$ (note that rhs of the previous inequality depends continuously on $\eps$ for any fixed $x,u,t$), we obtain
\begin{eqnarray*}
\|\phi(t,x,u)\|_X &\leq& \rho\Big(2\beta\big(\xi(\| x \|_X + \|u\|_{\Uc}),t\big) + 2 \gamma \circ \eta\big( \|u\|_{\Uc}\big)\Big)\\
&\leq& \rho\Big(2\beta\big(\xi(2\| x \|_X) + \xi(2\|u\|_{\Uc}),t\big) + 2 \gamma \circ \eta\big( \|u\|_{\Uc}\big)\Big)\\
&\leq& \rho\Big(2\beta\big(2\xi(2\| x \|_X),t\big) + 2\beta\big(2\xi(2\|u\|_{\Uc}),t\big) + 2 \gamma \circ \eta\big( \|u\|_{\Uc}\big)\Big)\\
&\leq& \rho\Big(2\beta\big(2\xi(2\| x \|_X),t\big) + 2\beta\big(2\xi(2\|u\|_{\Uc}),0\big) + 2 \gamma \circ \eta\big( \|u\|_{\Uc}\big)\Big)\\
&\leq& \rho\Big(4\beta\big(2\xi(2\| x \|_X),t\big)\Big) + \rho\Big(4\beta\big(2\xi(2\|u\|_{\Uc}),0\big) + 4 \gamma \circ \eta\big( \|u\|_{\Uc}\big)\Big)\\
&=:& \hat\beta\big(\| x \|_X,t\big) + \hat\gamma\big( \|u\|_{\Uc}\big),
\end{eqnarray*}
where $\hat\beta:(r,t) \mapsto \rho\Big(4\beta\big(2\xi(2r),t\big)\Big)$ is a $\KL$ function and $\hat\gamma(r):r \mapsto \rho\Big(4\beta\big(2\xi(2r),0\big) + 4 \gamma \circ \eta\big(r\big)\Big)$ is a $\Kinf$ function.
This shows ISS of $(X,\Uc,\phi)$. 

If $\rho$ and $\xi$ are linear, and $\Sigma$ is exp-ISS w.r.t. inputs in $\Uc_{c}$, then above argument justifies exp-ISS of $\Sigma$.
Clearly, if $\eta$ and $\gamma$ are linear, then also $\tilde\gamma$ is linear.
\end{proof}

\section{Input-to-state stability of nonlinear parabolic equations}
\label{ISS_nonlinear_parabolic_PDEs}

In this section we apply results from Section~\ref{sec:ISS_monotone_control_systems}
 to nonlinear parabolic equations with boundary inputs.

Let $G\subset \mathbb{R} ^{n} $ be an open bounded region, let $T>0$ be a constant and denote $D:=(0,T)\times G$. 
Let $CL(D)$ denote the class of functions $x\in C^{0} \left(\overline{D}\right)\bigcap C^{1,2} (D)$.
Denote for each $t\in[0,T]$ and each $x\in CL(D)$ a function $x[t]:\R_+ \to C^0(\overline{G})$ by $x[t]:=x(t,\cdot)$.

Consider the operator $L$ defined for a function $x\in CL(D)$ by
\begin{equation} 
\label{GrindEQ__1_}
 (Lx)(t,z):=\frac{\partial \, x}{\partial \, t} (t,z)-\sum _{i,j=1}^{n}a_{i,j} (z)\frac{\partial ^{2} \, x}{\partial \, z_{i} \, \partial \, z_{j} } (t,z) -f\big(z,x(t,z),\nabla x(t,z)\big), \ (t,z)\in D,
\end{equation} 
where $a_{i,j} \in C^{0} (\overline{G})$ for $i,j=1,...,n$ and $f:G\times \mathbb{R} \times \mathbb{R} ^{n} \to \mathbb{R} $ is a continuous function. The operator $L$ is called uniformly parabolic, if there exists a constant $K>0$ so that for all $\xi =(\xi _{1} ,...,\xi _{n} )\in \mathbb{R} ^{n} $ it holds that
\begin{equation} \label{GrindEQ__2_}
\sum _{i,j=1}^{n}a_{i,j} (z)\xi _{i} \xi _{j}  \ge K\left|\xi \right|^{2}  \quad \mbox{for all} \ z\in \overline{G}.
\end{equation}

We need the following modification of a classical comparison principle from \cite[Theorem 16, p. 52]{Fri83}.
\begin{proposition}
\label{prop:New_Comparison_Principle}
Let $L$ be uniformly parabolic. Assume that for every bounded set $W\subset \mathbb{R} $ there is a constant $k>0$ such that for every $w_1,w_2\in W$, $z\in G$, $\xi \in \mathbb{R} ^{n}$ with $w_1>w_2$ it holds that
\begin{equation} 
\label{GrindEQ__3_} 
f(z,w_1,\xi )-f(z,w_2,\xi ) < k(w_1-w_2).
\end{equation} 
Let $x,y\in CL(D)$ be so that
\begin{eqnarray}
(Ly)(t,z)&\ge& 0\ge (Lx)(t,z), \quad \mbox{for all} \ (t,z)\in D,\label{GrindEQ__4_}\\
y(0,z)&\ge& x(0,z), \quad \mbox{for all} \ z\in G, \label{GrindEQ__5_}\\
y(t,z)&\ge& x(t,z), \quad \mbox{for all} \ (t,z)\in [0,T]\times \partial G. \label{GrindEQ__6_}
\end{eqnarray}
Then $y(t,z)\ge x(t,z)$ for all $(t,z)\in \overline{D}$.
\end{proposition}

\begin{proof}
Since $y\in C^{0} \left(\overline{D}\right)$ and since $\overline{D}$ is compact, there exist constants $m\le M$ such that
\begin{equation} 
\label{GrindEQ__7_}
m\le y(t,z)\le M, \quad \mbox{for all} \ (t,z)\in \overline{D}.
\end{equation}

By assumption, there exists a constant $k>0$ such that for every $w_1,w_2\in [m,M+1]$, $(t,z)\in D$, $\xi \in \mathbb{R} ^{n} $ with $w_1>w_2$ inequality \eqref{GrindEQ__3_} holds. Consequently, it follows that for every $(t,z)\in D$ and for every $\varepsilon \in \left(0,e^{-kT}\right]$ the following inequality holds:
\begin{equation} \label{GrindEQ__8_} 
f\big(z,y(t,z)+\varepsilon e^{kt},\nabla y(t,z)\big)-ke^{kt}\varepsilon <f\big(z,y(t,z),\nabla y(t,z)\big). 
\end{equation} 

Next consider the functions $u,v\in CL(D)$ defined for all $(t,z)\in \overline{D}$ by
\begin{equation} \label{GrindEQ__9_}
\begin{array}{l} {u(t,z)=e^{-kt}y(t,z), \qquad v(t,z)=e^{-kt}x(t,z).} 
\end{array}                             
\end{equation}

It follows from the definitions \eqref{GrindEQ__9_} and relations \eqref{GrindEQ__4_}, \eqref{GrindEQ__5_}, \eqref{GrindEQ__6_} that the following inequalities hold:
\begin{eqnarray}
u(0,z)&\ge& v(0,z), \quad \mbox{for all} \ z\in G, \label{GrindEQ__10_}\\
u(t,z)&\ge& v(t,z), \quad \mbox{for all} \ (t,z)\in [0,T]\times \partial G. \label{GrindEQ__11_}
\end{eqnarray}
Using \eqref{GrindEQ__1_} and the assumption \eqref{GrindEQ__4_}, we obtain that for all $(t,z)\in D$ it holds that
\begin{eqnarray}
\label{GrindEQ__12_}
\frac{\partial \, u}{\partial \, t} (t,z) &=& -k u(t,z) + e^{-kt} \frac{\partial \, y}{\partial \, t} (t,z) \nonumber\\
&\geq & e^{-kt} \sum _{i,j=1}^{n}a_{i,j} (z)\frac{\partial ^{2} \, y}{\partial \, z_{i} \, \partial \, z_{j} } (t,z) 
+e^{-kt}f\left(z,y(t,z),\nabla y(t,z)\right) -k u(t,z) \nonumber\\
& = & \sum _{i,j=1}^{n}a_{i,j} (z)\frac{\partial ^{2} \, u}{\partial \, z_{i} \, \partial \, z_{j} } (t,z) 
+e^{-kt}f\left(z,e^{kt}u(t,z),e^{kt}\nabla u(t,z)\right) -k u(t,z).
\end{eqnarray}

Analogously, we obtain for all $(t,z)\in D$ that
\begin{equation} 
\label{GrindEQ__13_}
\frac{\partial \, v}{\partial \, t} (t,z)\le \sum _{i,j=1}^{n}a_{i,j} (z)\frac{\partial ^{2} \, v}{\partial \, z_{i} \, \partial \, z_{j} } (t,z) +e^{-kt}f\left(z,e^{kt}v(t,z),e^{kt}\nabla v(t,z)\right)-kv(t,z).
\end{equation}                                              
Finally, define for every $\varepsilon \in \left(0,e^{-kT}\right]$ the function $u_{\varepsilon } \in CL(D)$ by means of the formula:
\begin{equation} \label{GrindEQ__14_}
u_{\varepsilon } (t,z)=u(t,z)+\varepsilon , \quad \mbox{for all} \ (t,z)\in \overline{D}.
\end{equation}
It follows from definition \eqref{GrindEQ__14_} and inequalities \eqref{GrindEQ__10_}, \eqref{GrindEQ__11_} that the following inequalities hold for all $\varepsilon \in \left(0,e^{-kT}\right]$:
\begin{eqnarray}
u_{\varepsilon } (0,z)&>&v(0,z), \quad \mbox{for all} \ z\in G, \label{GrindEQ__15_}\\
u_{\varepsilon } (t,z)&>&v(t,z), \quad \mbox{for all} \ (t,z)\in [0,T]\times \partial G. \label{GrindEQ__16_}
\end{eqnarray}
Using definitions \eqref{GrindEQ__9_}, \eqref{GrindEQ__14_} and inequalities \eqref{GrindEQ__8_}, \eqref{GrindEQ__12_} we get for all $\varepsilon \in \left(0,e^{-kT}\right]$ and all $(t,z)\in D$:
\begin{eqnarray}
\frac{\partial \, u_{\varepsilon } }{\partial \, t} (t,z) &=& \frac{\partial \, u}{\partial \, t} (t,z) \nonumber\\
&\geq& \sum _{i,j=1}^{n}a_{i,j} (z)\frac{\partial ^{2} \, u}{\partial \, z_{i} \, \partial \, z_{j} } (t,z) 
+e^{-kt}f\big(z,e^{kt}u(t,z),e^{kt}\nabla u(t,z)\big) -k u(t,z) \nonumber\\
&=& \sum _{i,j=1}^{n}a_{i,j} (z)\frac{\partial ^{2} \, u_{\varepsilon }}{\partial \, z_{i} \, \partial \, z_{j} } (t,z) 
{+}e^{-kt}f\big(z,e^{kt}u(t,z),e^{kt}\nabla u_{\varepsilon }(t,z)\big) {-}k (u_{\varepsilon }(t,z){-}\eps) \nonumber\\
&>& \sum _{i,j=1}^{n}a_{i,j} (z)\frac{\partial ^{2} \, u_{\varepsilon }}{\partial \, z_{i} \, \partial \, z_{j} } (t,z) 
+e^{-kt}\Big(f\big(z,e^{kt}u(t,z),e^{kt}\nabla u_{\varepsilon }(t,z)\big)-ke^{kt}\eps \Big) \nonumber\\
&&\qquad\qquad\qquad\qquad\qquad\qquad\qquad\qquad\qquad\qquad\qquad\qquad -k (u_{\varepsilon }(t,z)-\eps) \nonumber\\
\label{GrindEQ__17_} &=&  \sum _{i,j=1}^{n}a_{i,j} (z)\frac{\partial ^{2} \, u_{\varepsilon } }{\partial \, z_{i} \, \partial \, z_{j} } (t,z) +e^{-kt}f\big(z,e^{kt}u_{\varepsilon } (t,z),e^{kt}\nabla u_{\varepsilon } (t,z)\big)-ku_{\varepsilon } (t,z).
\end{eqnarray}

 Using inequalities \eqref{GrindEQ__13_}, \eqref{GrindEQ__15_}, \eqref{GrindEQ__16_}, \eqref{GrindEQ__17_} and the fact that $u_{\varepsilon } ,v\in CL(D)$, it follows from 
\cite[Theorem 16, p. 52]{Fri83} for every $\varepsilon \in \left(0,e^{-kT}\right]$ that:
\begin{equation} 
\label{GrindEQ__18_}
u_{\varepsilon } (t,z)>v(t,z), \quad \mbox{for all} \ (t,z)\in D.
\end{equation}

 Inequality $y(t,z)\ge x(t,z)$ for all $(t,z)\in \overline{D}$ is a direct consequence of definitions \eqref{GrindEQ__9_}, \eqref{GrindEQ__14_}, inequalities \eqref{GrindEQ__5_}, \eqref{GrindEQ__6_}, \eqref{GrindEQ__18_} and continuity of the functions $x,y\in CL(D)$ on $\overline{D}$. 
\end{proof}

Since our intention is to analyze forward complete systems, we introduce some more notation.
Let $CL$ denote the class of functions $x\in C^{0} \left(\mathbb{R} _{+} \times \overline{G}\right)\bigcap C^{1,2} ((0,+\infty )\times G)$.

Now we apply the established results to analyze the initial boundary value problem:
\begin{eqnarray}
(Lx)(t,z)&=&0, \quad \mbox{for all }(t,z)\in (0,+\infty )\times G, \label{GrindEQ__22_}\\
x(0,z)&=&x_{0} (z), \quad \mbox{for all }z\in G, \label{GrindEQ__23_}\\
x(t,z)&=&u(t,z), \quad \mbox{for all }(t,z)\in \mathbb{R} _{+} \times \partial G,  \label{GrindEQ__24_}
\end{eqnarray}
where $x_{0} \in C^{0} (\overline{G})$, $L$ is the uniformly parabolic operator defined by \eqref{GrindEQ__1_} with $f\in C^{0} (\overline{G}\times \mathbb{R} \times \mathbb{R} ^{n} )$.

In this section we assume that the space of input values is $U=C^{0} (\partial G)$, endowed with the standard sup-norm
and that 
$\mathcal{U}:=\left\{\, u\in C^{0} (\mathbb{R} _{+} ;U):\, \, u \mbox{ is bounded}\right\}$, endowed with
\begin{itemize}
  \item the partial order $\le $ for which $u\le v$ iff $u(t,z)\le v(t,z)$ for all $(t,z)\in \mathbb{R} _{+} \times \partial G$,
 \item the norm 
$\left\| u\right\| _{\mathcal{U}} =\mathop{\sup }\limits_{t\ge 0} \left\| u[t]\right\| _{U} =\mathop{\sup }\limits_{z\in \partial G,\ t\ge 0} \left|u(t,z)\right|$. 
\end{itemize}
In the sequel we will need also a subspace of $\Uc$ consisting of constant in time and space inputs:
\begin{eqnarray}
\Uc_{c}:=\{u \in \Uc:\exists k \in\R:\ u(t,z)=k \ \forall (t,z)\in \mathbb{R} _{+} \times \partial G\}.
\label{eq:Uc_const_def}
\end{eqnarray}

Define for $x\in C^0(G)$ the standard $L^p(G)$-norm as $\left\|x \right\| _{p}:=\Big(\int_G |x(z)|^p dz\Big)^{1/p}$.
The Euclidean distance between $w\in G$ and $W \subset \overline{G}$ is denoted by $\rho(w,W):=\inf_{s\in W}|w-s|$.

The following assumptions are instrumental in what follows.
\begin{itemize}
   \item[(H1)] There exists a linear space $X\subseteq C^{0} (\overline{G})$, containing the functions $\{x\in C^{0} (\overline{G}):\exists k\in\R \mbox{ s.t. } x(\cdot)=k\}$, such that for each $x_{0} \in X$ which is constant on $\partial G$, there exists a set of inputs $\Uc(x_{0} )\subseteq \left\{\, v\in \Uc\, :\, v(0,z)=x_{0} (z)\, \mbox{for}\, \, z\in \partial G\, \right\}$, which contains constant in time and space inputs
\[
\left\{\, v\in \Uc\, :\, v(t,z)=x_{0} (z)\, \mbox{for}\, \, z\in \partial G\, \mbox{and}\, \, t\ge 0\, \right\}
\] 
with the following property: for every $x_{0} \in X$, $u\in \Uc(x_{0} )$ there exists a solution $x\in CL$ of the initial boundary 
value problem \eqref{GrindEQ__22_}, \eqref{GrindEQ__23_}, \eqref{GrindEQ__24_} for which $x[t]\in X$ for all $t\ge 0$.  

 \item[(H2)] Assume that for every bounded set $W\subset \mathbb{R} $ there exists a constant $k>0$ such that for every $w_1,w_2\in W$, $(t,z)\in D$, $\xi \in \mathbb{R} ^{n} $ with $w_1>w_2$ inequality \eqref{GrindEQ__3_} holds. Moreover, the function $\overline{G}\times \mathbb{R} \times \mathbb{R} ^{n} \ni (z,w,\xi )\mapsto f(z,w,\xi )\in \mathbb{R} $ is continuously differentiable with respect to $w\in \mathbb{R} $ and $\xi \in \mathbb{R} ^{n} $.

 \item[(H3)] For every $\delta >0$, $a\in \mathbb{R} $, $x \in X$ there exists a continuous function $k:\overline{G}\to [0,1]$ with $k(z)=1$ for $z\in \partial G$, $k(z)=0$ for all $z\in G$ with $\rho(z,\partial G)\ge \delta $ and such that $f\in X$ where $f(z)=(1-k(z))x(z)+a\, k(z)$.
\end{itemize}

\begin{remark}\em
Note that (H3) is a condition on the geometry of the boundary of $G$ which is automatically satisfied when $G$ is an open interval in $\R$.
\end{remark}

We equip $X$ in (H1) with the partial order $\le $ for which $x\le y$ iff $x(z)\le y(z)$ for all $z\in \overline{G}$.
For existence theorems, which can be used to verify (H1), we refer to \cite[Chapters 3, 7]{Fri83}.

The next result assures monotonicity of the initial boundary value problem \eqref{GrindEQ__22_}, \eqref{GrindEQ__23_}, \eqref{GrindEQ__24_}.
\begin{theorem}
\label{thm:Monotonicity_Parabolic_Systems}
Suppose that assumptions (H1), (H2) hold and let $p\in[1,+\infty]$. Let us endow the linear space $X$ in (H1) with the standard $L^{p} (G)$-norm, which we denote by $\left\| x\right\| _{p}$.
Then:
\begin{itemize}
   \item[(i)] Initial boundary value problem \eqref{GrindEQ__22_}, \eqref{GrindEQ__23_}, \eqref{GrindEQ__24_} gives rise to
 the monotone control system $\Sigma=\left(X,\mathcal{U},\phi \right)$, where $\phi$ is the solution map of \eqref{GrindEQ__22_}, \eqref{GrindEQ__23_}, \eqref{GrindEQ__24_}.
 \item[(ii)] If additionally (H3) holds, then conditions (i) and (ii) of Proposition~\ref{prop:ISS_Monotone_states_inputs_Systems} hold with $\rho$, $\eta$, $\zeta$ being linear functions and $\Uc_{c}$ given by \eqref{eq:Uc_const_def}.  
\end{itemize}
\end{theorem}

\begin{proof}
We start by proving claim \textit{(i)}.
Pick any $x_{0} \in X$, $u\in \Uc(x_{0})$. Uniqueness of a corresponding solution of the initial boundary value problem \eqref{GrindEQ__22_}, \eqref{GrindEQ__23_}, \eqref{GrindEQ__24_} (which we denote by $\phi(\cdot,x_0,u)$) follows from the fact that the function $\overline{G}\times \mathbb{R} \times \mathbb{R} ^{n} \ni (z,x,\xi )\mapsto f(z,x,\xi )\in \mathbb{R} $ is continuously differentiable with respect to $(x,\xi )\in \mathbb{R} \times \mathbb{R} ^{n} $ and 
\cite[Theorem 8, p. 41]{Fri83}.

Exploiting uniqueness, one can directly show that properties ($\Sigma$\ref{axiom:Causality}) and ($\Sigma$\ref{axiom:Cocycle}) of Definition 2.1 hold for a triple $\Sigma=(X,\Uc,\phi)$. 
Therefore, $\Sigma$ is a control system. 
The fact that $\Sigma$ is a monotone control system is a direct consequence of Assumption (H2), the fact that $x\in CL$ and Proposition~\ref{prop:New_Comparison_Principle}.

\textit{(ii).} Assume that (H3) holds. Next we show that conditions (i) and (ii) of Proposition~\ref{prop:ISS_Monotone_states_inputs_Systems} hold.

Let $x, x_{+}, x_{-} \in X$ with $x_{-} (z)\le x(z)\le x_{+} (z)$ for all $z\in \overline{G}$ be given. 
Therefore, we get $\left|x(z)\right|\le \left|x_{-} (z)\right|+\left|x_{+} (z)\right|$ for all $z\in \overline{G}$. 
Due to Minkowski's inequality, 
\[
\left\| x \right\| _{p} \le \left\| x_{-} \right\| _{p} +\left\| x_{+} \right\| _{p}, \text{ for every } 1\le p\le \infty. 
\]
Therefore, condition (i) of Proposition~\ref{prop:ISS_Monotone_states_inputs_Systems} holds with $\rho (s):=s$ for $s\ge 0$. 

Let $\varepsilon >0$, $x\in X$ and $u\in \Uc(x)$ be given. Define:
\begin{equation} \label{GrindEQ__25_}
u_{-} (t,z):=-\left\| u\right\| _{\mathcal{U}} -\varepsilon,\qquad u_{+} (t,z):=\left\| u\right\| _{\mathcal{U}} +\varepsilon , \quad \mbox{for all }(t,z)\in \mathbb{R} _{+} \times \partial G.              
\end{equation}

 Since $\left\| u\right\| _{\mathcal{U}} =\mathop{\sup }\limits_{z\in \partial G,\ t\ge 0} \left|u(t,z)\right|$, it follows that $u_{-} \le u\le u_{+} $. Definitions \eqref{GrindEQ__25_} imply that inequality \eqref{GrindEQ__20_} holds with $\eta (s):=s$ for $s\ge 0$.

 Since $u\in \Uc(x)$ and since $\Uc(x)\subseteq \left\{v\in C^{0} (\mathbb{R} _{+} \times \partial G)\, :\, v(0,z)=x(z)\, \mbox{for}\, \, z\in \partial G \right\}$, it follows that $u(0,z)=x(z)$ for all $z\in \partial G$. Thus we have:
\begin{equation} \label{GrindEQ__26_}
-\left\| u\right\| _{\mathcal{U}} \le u(0,z)=x(z)\le \left\| u\right\| _{\mathcal{U}} , \quad \mbox{for all }z\in \partial G.                               
\end{equation}

It follows from \eqref{GrindEQ__26_}, compactness of $\overline{G}$ and continuity of $x\in X$ that for every $\varepsilon >0$ there exists $\delta >0$ such that
\begin{equation} \label{GrindEQ__27_} 
z\in G \ , \ \rho(z,\partial G)<\delta \quad \Rightarrow \quad -\left\| u\right\| _{\mathcal{U}} -\varepsilon \le x(z)\le \left\| u\right\| _{\mathcal{U}} +\varepsilon.
\end{equation}

 By virtue of Assumption (H3), there exists $k:\overline{G}\to [0,1]$ with $k(z)=1$ for $z\in \partial G$, $k(z)=0$ for all $z\in G$ with $\rho(z,\partial G)\ge \delta $ and such that $x_{-} \in X$, where 
\begin{equation} \label{GrindEQ__28_} 
x_{-} (z):=\big(1-k(z)\big)x(z)-\left(\left\| u\right\| _{\mathcal{U}} +\varepsilon \right)\, k(z). 
\end{equation} 

Implication \eqref{GrindEQ__27_}, definition \eqref{GrindEQ__28_}, the fact that $k(z)=0$ for $z\in G$ with $\rho(z,\partial G)\ge \delta $  and the fact that $k(z)\ge 0$ for all $z\in \overline{G}$ imply that 
\[
x_-(z) \leq \big(1-k(z)\big)x(z) + x(z) k(z) = x(z),\quad z\in \overline{G}.
\]
Moreover, definition \eqref{GrindEQ__28_} in conjunction with the fact that $k(z)\in [0,1]$ for all $z\in \overline{G}$ implies that
\begin{equation} \label{GrindEQ__29_} 
\left\| x_{-} \right\| _{p} \le \left\| x\right\| _{p} +\big(\left\| u\right\| _{\mathcal{U}} +\varepsilon \big)(\mu(G))^{1/p}, \mbox{ for every}\ 1\le p<\infty
\end{equation}
where $\mu(G)$ is the Lebesgue measure of $G$. Also it holds that
\begin{equation} \label{GrindEQ__30_} 
\left\| x_{-} \right\| _{\infty } \le \left\| x\right\| _{\infty } +\left\| u\right\| _{\mathcal{U}} +\varepsilon.
\end{equation} 
Furthermore, definitions \eqref{GrindEQ__25_}, \eqref{GrindEQ__28_} and the fact that $k(z)=1$ for $z\in \partial G$ imply that $u_{-} (t,z)=x_{-} (z) = -(\|u\|_{\Uc}+\eps)$ for all $(t,z)\in \mathbb{R} _{+} \times \partial G$. By virtue of Assumption (H1) we have 
\[
\left\{\, v\in C^{0} (\mathbb{R} _{+} \times \partial G)\, :\, v(t,z)=x_{-} (z)\, \mbox{for}\, \, z\in \partial G\, \mbox{and}\, \, t\ge 0\, \right\}\subseteq \Uc(x_{-} )
\]  
which implies $u_{-} \in \Uc(x_{-} )$. 

Analogously, there exists $l:\overline{G}\to [0,1]$ with $l(z)=1$ for $z\in \partial G$, $l(z)=0$ for all $z\in G$ with $\rho(z,\partial G)\ge \delta $ and such that $x_{+} \in X$, where 
\begin{equation} \label{GrindEQ__31_} 
x_{+} (z):=\big(1-l(z)\big)x(z)+\left(\left\| u\right\| _{\mathcal{U}} +\varepsilon \right)\, l(z).
\end{equation} 
and  $x_+$ satisfies the estimates
\begin{equation} \label{GrindEQ__32_}
\left\| x_{+} \right\| _{p} \le \left\| x\right\| _{p} +\left(\left\| u\right\| _{\mathcal{U}} +\varepsilon \right)(\mu(G))^{1/p},\quad \mbox{for every} \ 1\le p<\infty
\end{equation} 
and
\begin{equation} \label{GrindEQ__33_} 
\left\| x_{+} \right\| _{\infty } \le \left\| x\right\| _{\infty } +\left\| u\right\| _{\mathcal{U}} +\varepsilon,
\end{equation} 
and as above one verifies that $u_{+} \in \Uc(x_{+} )$.


Inequalities \eqref{GrindEQ__29_}, \eqref{GrindEQ__30_}, \eqref{GrindEQ__32_}, \eqref{GrindEQ__33_} imply that \eqref{GrindEQ__21_} holds with $\xi (s):=\big(1+(\mu(G))^{1/p} \big) s$ for $1\le p<\infty $, $s\ge 0$ and $\zeta (s):= s$ for $p=\infty $, $s\ge 0$.  
The proof is complete. 
\end{proof}

We continue to assume that the axioms (H1) and (H2) hold and that $X$ is as in (H1).

Consider now the following equations: 
\begin{eqnarray}
\frac{\partial y}{\partial t}(t,z) - \sum_{i,j=1}^n a_{ij}(z)\frac{\partial^2 y}{\partial z_i \partial z_j}(t,z) -
f\Big(z,y(t,z)+v(t,z), \nabla y(t,z)\Big) =0,
\label{eq:Nonlinear_Parabolic_Equation_transformed}
\end{eqnarray}
where $t >0$, $z \in G$, together with homogeneous Dirichlet boundary conditions
\begin{eqnarray}
y(t,z) = 0, \quad z \in \partial G,\ t \geq 0.
\label{eq:Dirichlet_transformed}
\end{eqnarray}
The state space of \eqref{eq:Nonlinear_Parabolic_Equation_transformed} is a linear space
\[
Y:=\{y\in X: y(z)=0,\ z\in\partial{G}\}
\]
and the input $v$ belongs to the space of constant in time and space inputs
\[
\Vc:=\{v\in C^0(\R_+\times \overline{G},\R): \exists k\in\R: v(t,z)=k,\ (t,z)\in\R_+\times \overline{G}\}.
\]
Denote the solutions of \eqref{eq:Nonlinear_Parabolic_Equation_transformed}, \eqref{eq:Dirichlet_transformed}, corresponding to the initial condition $y \in Y$ and input $v$ by $\phi_y(t,y,v)$.
It is easy to see that for any $x\in X$ and any $v \in \Vc$ for which $v|_{\partial G} \in \Uc(x)$  it holds that
\begin{eqnarray}
\phi_y(t,y,v)=\phi(t,x,v|_{\partial G}) -v,\quad \mbox{where}\ y=x-v,
\label{eq:phi_phyy_relation}
\end{eqnarray}
and $y \in Y$ since $v\in X$, $X$ is a linear space and $x(z)=v(z)$ for $z\in\partial G$.

Now we are able to prove our main result, showing that ISS of nonlinear parabolic systems w.r.t. a boundary input can be reduced to the problem of ISS of a parabolic system with a distributed and constant input, which is conceptually much simpler.
\begin{theorem}
\label{thm:ISS_Parabolic_Systems_Characterization}
Suppose that assumptions (H1), (H2) and (H3) hold and let $p\in[1,+\infty]$. Let us endow the linear spaces $X,Y$ with the norm $\left\| \cdot \right\| _{p}$. The following statements are equivalent:

\begin{itemize}
 \item[(i)] The system \eqref{GrindEQ__22_}, \eqref{GrindEQ__23_}, \eqref{GrindEQ__24_}  with the state space $X$ is ISS w.r.t. inputs of class $\Uc$.
 \item[(ii)] The system \eqref{GrindEQ__22_}, \eqref{GrindEQ__23_}, \eqref{GrindEQ__24_} with the state space $X$ is ISS w.r.t. constant in time and space inputs of class $\Uc$.
 \item[(iii)] The system \eqref{eq:Nonlinear_Parabolic_Equation_transformed}, \eqref{eq:Dirichlet_transformed} 
 with the state space $Y$ is ISS w.r.t. inputs in $\Vc$.
\end{itemize}
\end{theorem}

\begin{proof}
(i) $\Iff$ (ii). Follows by Theorem~\ref{thm:Monotonicity_Parabolic_Systems} and Proposition~\ref{prop:ISS_Monotone_states_inputs_Systems}.

(iii) $\Rightarrow$ (i).
Pick any $x\in X$ and a constant in time and space input $u \in \Uc(x)$, with $u(t,z) = k$ for a certain $k\in\R$ and all $t\ge0$, $z\in\partial G$. 
Define $v(t,z):=k$ for all $t\ge 0$, $z\in \overline{G}$. Then $v|_{\partial G} = u \in \Uc(x)$ and \eqref{eq:phi_phyy_relation} holds.

ISS of \eqref{eq:Nonlinear_Parabolic_Equation_transformed}, \eqref{eq:Dirichlet_transformed} with state space $Y$ w.r.t. inputs in $\Vc$ and equation 
\eqref{eq:phi_phyy_relation} ensures that
\begin{eqnarray*}
\|\phi(t,x,u) - v\|_p \leq \beta(\|x-v\|_p,t) + \gamma(\|v\|_\Vc).
\end{eqnarray*}
Since $\|v\|_p \leq (\mu(G))^{1/p}\|v\|_\Vc = (\mu(G))^{1/p}\|u\|_\Uc$, we proceed to
\begin{eqnarray*}
\|\phi(t,x,u)\|_p &\leq& \|\phi(t,x,u) - v\|_p + \|v\|_p \\
&\leq& \beta(\|x-v\|_p,t) + \gamma(\|u\|_\Uc) + (\mu(G))^{1/p}\|u\|_\Uc \\
&\leq& \beta(\|x\|_p + \|v\|_p,t) + \gamma(\|u\|_\Uc) + (\mu(G))^{1/p}\|u\|_\Uc\\
&\leq& \beta(\|x\|_p + (\mu(G))^{1/p}\|u\|_\Uc,t) + \gamma(\|u\|_\Uc) + (\mu(G))^{1/p}\|u\|_\Uc.
\end{eqnarray*}

By means of the trivial inequality $\beta(a+b,t)\leq \beta(2a,t) + \beta(2b,t)$, which holds for any $a,b,t\in\R_+$, we proceed to
\begin{eqnarray*}
\|\phi(t,x,u)\|_p  &\leq& \beta(2\|x\|_p,t) + \beta(2(\mu(G))^{1/p}\|u\|_{\Uc},t) + \gamma(\|u\|_{\Uc}) + (\mu(G))^{1/p}\|u\|_{\Uc} \\
&\leq& \beta(2\|x\|_p,t) + \tilde\gamma(\|u\|_{\Uc}),
\end{eqnarray*}
where $\tilde\gamma(r):= \beta(4(\mu(G))^{1/p}r,0) + \gamma(r) + (\mu(G))^{1/p}r$. Clearly, $\tilde\gamma\in\Kinf$.

This shows that the system \eqref{GrindEQ__22_}, \eqref{GrindEQ__23_}, \eqref{GrindEQ__24_} is ISS.

(ii) $\Rightarrow$ (iii).
Let the system \eqref{GrindEQ__22_}, \eqref{GrindEQ__23_}, \eqref{GrindEQ__24_} be ISS for constant inputs.
Then there exist $\beta \in \KL$ and $\gamma \in \Kinf$ so that for all $t\geq0$, all $x\in X$ and all constant $u\in\Uc(x)$ the estimate
\eqref{eq:ISS_estimate} holds.
Pick any $v \in \Vc$, which we consider as an input to the system \eqref{eq:Nonlinear_Parabolic_Equation_transformed}, \eqref{eq:Dirichlet_transformed}. Pick also any initial state $y\in Y$. Then $(x,u):=(y+v,v)$ is an admissible pair for the system \eqref{GrindEQ__22_}, \eqref{GrindEQ__23_}, \eqref{GrindEQ__24_}. ISS of \eqref{GrindEQ__22_}, \eqref{GrindEQ__23_}, \eqref{GrindEQ__24_} for constant inputs together with \eqref{eq:phi_phyy_relation} leads to
\begin{eqnarray*}
\| \phi_y(t,y,v) + v \|_p \leq \beta(\| y+v \|_p,t) + \gamma( \|v|_{\partial G}\|_{\Uc}).
\end{eqnarray*}
Thus,
\begin{eqnarray*}
\| \phi_y(t,y,v) \|_p &\leq& \| \phi_y(t,y,v) + v \|_p + \|v\|_p\\
&=& \| \phi(t,y+v,v|_{\partial G})\|_p + \|v\|_p\\
&\leq& \beta(\| y+v \|_p,t) + \gamma( \|v|_{\partial G}\|_{\Uc}) + (\mu(G))^{1/p}\|v\|_\Vc \\
&\leq& \beta(\| y\|_p + \|v \|_p,t) + \gamma( \|v\|_\Vc) + (\mu(G))^{1/p}\|v\|_\Vc \\
&\leq& \beta(2\| y\|_p,t)+\beta(2 \|v \|_p,t) + \gamma( \|v\|_\Vc) + (\mu(G))^{1/p}\|v\|_\Vc \\
&\leq& \beta(2\| y\|_p,t)+\beta(2 (\mu(G))^{1/p} \|v \|_\Vc,0) + \gamma( \|v\|_\Vc) + (\mu(G))^{1/p}\|v\|_\Vc \\
&=& \beta(2\| y\|_p,t)+\tilde\gamma_2( \|v\|_\Vc),
\end{eqnarray*}
with $\tilde\gamma_2(r):= \beta(2(\mu(G))^{1/p}r,0) + \gamma(r) + (\mu(G))^{1/p}r$. This shows implication  (ii) $\Rightarrow$ (iii).
\end{proof}

The same argument justifies the following result on exp-ISS property:
\begin{theorem}
\label{thm:exp-ISS_Parabolic_Systems_Characterization}
Suppose that assumptions (H1), (H2) and (H3) hold and let $p\in[1,+\infty]$. Let us endow the linear spaces $X,Y$ with the norm $\left\| \cdot \right\| _{p}$.
The following statements are equivalent:

\begin{itemize}
 \item[(i)] The system \eqref{GrindEQ__22_}, \eqref{GrindEQ__23_}, \eqref{GrindEQ__24_}  with the state space $X$ is exp-ISS w.r.t. inputs of class $\Uc$.
 \item[(ii)] The system \eqref{GrindEQ__22_}, \eqref{GrindEQ__23_}, \eqref{GrindEQ__24_} with the state space $X$ is exp-ISS w.r.t. constant in time inputs of class $\Uc$.
 \item[(iii)] The system \eqref{eq:Nonlinear_Parabolic_Equation_transformed}, \eqref{eq:Dirichlet_transformed} 
 with the state space $Y$ is exp-ISS w.r.t. inputs in $\Vc$.
\end{itemize}
\end{theorem}

\begin{proof}
(i) $\Iff$ (ii). Follows by Theorem~\ref{thm:Monotonicity_Parabolic_Systems} (here linearity of $\rho$, $\eta$, $\zeta$ is important) and exp-ISS part of Proposition~\ref{prop:ISS_Monotone_states_inputs_Systems}.

(ii) $\Iff$ (iii). Along the lines of the proof of Theorem~\ref{thm:ISS_Parabolic_Systems_Characterization}.
\end{proof}

\section{Applications}

In this section, we apply the above results to several problems of specific interest.
We continue to assume that $G\subset \R^n$ is an open connected and bounded set with the smooth boundary
and $D:=(0,+\infty)\times G$.

\subsection{ISS of linear parabolic systems with boundary inputs}
\label{sec:Lin_Parabolic_Systems}

Consider the linear heat equation with a Dirichlet boundary input:
\begin{equation}
\label{eq:IBVP_Operator_L_Linear}
\begin{array}{l}
\frac{\partial x}{\partial t}(t,z) = \Delta x(t,z) + ax(t,z),\quad (t,z) \in D,  \\
x(t,z) = u(t,z), \quad (t,z)\in (0,+\infty) \times \partial G,\\
x(0,z) = x_0(z), \quad z \in \overline{G},
\end{array}
\end{equation}
where $\Delta$ is a Laplacian and $a \in \R$. 

The input $u:\R_+ \times\partial G\to \R$ is the trace of a function $\nu\in C^{0}(\R_+\times\overline{G}) \bigcap C^{1,2}((0,+\infty)\times G)$. 
Defining the function
\begin{eqnarray}
f(t,z):=\Delta \nu(t,z) - \frac{\partial \nu}{\partial t}(t,z),\quad t\geq0,\ z\in G
\label{eq:f_def}
\end{eqnarray}
and using the transformation
\begin{eqnarray}
x(t,z)=e^{at}\big(y(t,z)+\nu(t,z)\big),\quad t\geq 0 ,\ z\in \overline{G}
\label{eq:Aux_Transformation}
\end{eqnarray}
we are in a position to study an equivalent to \eqref{eq:IBVP_Operator_L_Linear} initial boundary value problem
\begin{equation}
\label{eq:IBVP_Operator_L_Linear_modified}
\begin{array}{l}
\frac{\partial y}{\partial t}(t,z) = \Delta y(t,z) + f(t,z),\quad (t,z) \in D, \\
y(t,z) = 0, \quad (t,z)\in (0,+\infty) \times \partial G,\\
y(0,z) = x(0,z)-\nu(0,z), \quad z \in \overline{G}.
\end{array}
\end{equation}
Let $m>0$ be the smallest integer for which $m\geq \frac{1+[n/2]}{2}$. \cite[Theorem 6, p. 365]{Eva98} in conjunction with \cite[Theorem 4, p. 288]{Eva98}
and \cite[Theorem 6, p. 270]{Eva98} guarantees that if 
\begin{itemize}
 \item[(p-i)] $y[0]\in H^{2m+1}(G)$, $\frac{d^k}{dt^k}(f[t]) \in L^2(0,T;H^{2m-2k}(G))$ for every $k=0,\ldots,m$ and $T>0$,
 \item[(p-ii)] $g_i \in H^1_0(G)$ for $i=0,\ldots,m$, where $g_0:=y[0]$, $g_m:=\frac{d^{m-1}}{dt^{m-1}}f[0]+\Delta g_{m-1}$,
\end{itemize}
then the initial boundary value problem \eqref{eq:IBVP_Operator_L_Linear_modified} has a unique solution $y\in CL$.

Therefore, using the transformation \eqref{eq:Aux_Transformation}, we conclude that for every $x_0\in H^{2m+1}(G)$ and for every input $u\in \R_+\times \partial G \to \R$ being the trace of a function $\nu\in C^{0}(\R_+\times\overline{G}) \bigcap C^{1,2}((0,+\infty)\times G)$ and satisfying (p-i), (p-ii), 
also the initial boundary value problem \eqref{eq:IBVP_Operator_L_Linear} has a unique solution $x\in CL$.
Hence, \eqref{eq:IBVP_Operator_L_Linear} defines a control system with
\begin{itemize}
 \item $X:=H^{2m+1}(G)$ with $\|\cdot\|_p$-norm, for any fixed $p\in[1,+\infty]$.
 \item  $\Uc(x)$ being the set of all inputs $u:\R_{+} \times \partial G\to \R$ which are traces of functions $\nu\in C^{0}(\R_+\times\overline{G}) \bigcap C^{1,2}((0,+\infty)\times G)$ satisfying $\frac{d^{k} }{d\, t^{k} } \left(f[t]\right)\in L^{2} \big(0,T;H^{2m-2k} (G)\big)$ for every $k=0,...,m$ and $T>0$ and $g_{i} \in H_{0}^{1} (G)$ for $i=0,...,m$, where $f$ is defined by 
 \eqref{eq:f_def}, $g_{0} :=x-v[0]$, $g_{1} :=f[0]+\Delta g_{0} $,\dots , $g_{m} :=\frac{d^{m-1} }{d\, t^{m-1} } \left(f[0]\right)+\Delta g_{m-1} $,

 \item $\phi(t,x_0,u)$ being the unique solution $x[t]$ of the initial boundary value problem \eqref{eq:IBVP_Operator_L_Linear}.
\end{itemize}
Notice that $X$ contains the constant functions. We conclude that $\Sigma$ satisfies (H1). Clearly, (H2) holds for $\Sigma$ as well.

Hence we obtain from Theorem~\ref{thm:exp-ISS_Parabolic_Systems_Characterization}:
\begin{corollary}
\label{cor:ISS_Linear_Heat_Equation}
Assume that $G \subset \R^n$ is an open bounded set with a smooth boundary for which Assumption (H3) holds. 
Then $\Sigma=(X,\Uc,\phi)$ is exp-ISS with a linear gain function iff $\Sigma$ is 0-UGAS.
\end{corollary}

\begin{proof}
Clearly, if $\Sigma=(X,\Uc,\phi)$ is exp-ISS, then $\Sigma=(X,\Uc,\phi)$ is 0-UGAS.

Now assume that $\Sigma=(X,\Uc,\phi)$ is 0-UGAS and let us prove the converse implication. 
\eqref{eq:IBVP_Operator_L_Linear} is a problem \eqref{GrindEQ__1_}, corresponding to the operator
$Lx:= \frac{\partial x}{\partial t} - \Delta x - ax$.

According to Theorem~\ref{thm:exp-ISS_Parabolic_Systems_Characterization}, exp-ISS of \eqref{eq:IBVP_Operator_L_Linear} is equivalent to exp-ISS of the system
\begin{eqnarray}
\frac{\partial y}{\partial t} = \Delta y +ay+av,\quad t >0, \ z \in G,
\label{eq:Nonlinear_Parabolic_Equation_transformed-2}
\end{eqnarray}
with homogeneous Dirichlet boundary condition \eqref{eq:Dirichlet_transformed} and constant inputs $v \in \R$.

In order to prove the claim, we are going to use the semigroup approach. Consider three cases: $p\in (1,+\infty)$, $p=1$ and $p=+\infty$.
The operator $A_p$, $p\in (1,+\infty)$ with a domain of definition $D(A_p):=W^{2,p}(G) \bigcap W^{1,p}_0(G)$ and $A_p x:=\Delta x + ax$, for $x\in D(A_p)$ generates an analytic semigroup over $L^p(G)$, see \cite[Theorem 3.6, p. 215]{Paz83}.

For $p=1$ the operator $A_1$ with a domain of definition $D(A_1):= \{x\in W^{1,1}_0(G):\Delta x \in L^1(G)\}$ and $A_1 x:=\Delta x + ax$, for $x\in D(A_1)$ generates an analytic semigroup over $L^1(G)$, see \cite[Theorem 3.10, p. 218]{Paz83}.
Analogously, one can define an operator $A_\infty:=\Delta + aI$ with a certain domain of definition $D(A_\infty)$, which generates an analytic semigroup over $C_0(\overline{G}):=\{x\in C(\overline{G}): x(z)=0 \ \forall z\in\partial G\}$, see \cite[Theorem 3.7, p. 217]{Paz83}.

Now, since we assume that $\Sigma$  is 0-UGAS in the norm $\|\cdot \|_p$, then also the operator $A_p$ generates 0-UGAS (exponentially stable) $C_0$-semigroup $T_p$, which follows since $X$ is dense in the $L^p(G)$ for $p\in[1,\infty)$ and in $C_0(\overline{G})$ for $p=\infty$ and since $T_p$ is a semigroup of bounded operators (hence the norm of the operator $T_p(t)$ with a domain restricted to $Y$ is equal to the norm of $T_p(t)$ as an operator on $X$).


Since $T_p$ is an exponentially stable semigroup, \cite[Proposition 3]{DaM13} ensures that \eqref{eq:Nonlinear_Parabolic_Equation_transformed-2} is exp-ISS with a linear gain function. 
An inspection of Theorem~\ref{thm:exp-ISS_Parabolic_Systems_Characterization} shows that \eqref{eq:IBVP_Operator_L_Linear} is exp-ISS with a linear gain function as well.
%
%
%
\end{proof}

\begin{remark}\em
\label{rem:Linfinity}
Note that the operator $\Delta + aI$ in the above proof does not generate a strongly continuous semigroup over the space $L^\infty(G)$, see \cite[p.217]{Paz83} or \cite[Lemma 2.6.5, Remark 2.6.6]{CaH98}, therefore it is of importance to define $A_\infty$ as a generator of strongly continuous semigroup over $C_0(\overline{G})$.
\end{remark}

\begin{remark}\em
\label{rem:Practical_applications}
There are different ways to ensure that $\Sigma$ is 0-UGAS. For $p\in(1,+\infty)$ a usual way would be to construct a Lyapunov functional. 
On the other hand, since $T_p$ is an analytic semigroup for any $p\in[1,+\infty]$,
$T_p$ is exponentially stable (i.e. 0-UGAS) if and only if the spectrum of $A_p$ lies in $\{z\in \C:\text{Re}z<0\}$, see e.g. \cite[p.387]{Tri75}.
\end{remark}
%

\subsection{ISS stabilization of 1-D parabolic systems via backstepping. }

Consider the initial-boundary value problem
\begin{eqnarray}
\frac{\partial \, x}{\partial \, t} (t,z)=a\frac{\partial ^{2} \, x}{\partial \, z^{2} } (t,z),\quad (t,z)\in (0,+\infty )\times (0,1),
\label{GrindEQ__47_}
\end{eqnarray}
where $a>0$ is a constant, subject to the boundary conditions
\begin{eqnarray}
x(t,0)-d_{0} (t)=x(t,1)-d_{1} (t)=0,\quad t\ge 0,
\label{GrindEQ__48_}
\end{eqnarray}
where $d_{0} ,d_{1} :\R _{+} \to \R $ are given boundary inputs and the initial condition
\begin{eqnarray}
 x(0,z)=x_{0} (z),\quad z\in [0,1],
\label{GrindEQ__49_}
\end{eqnarray}
where $x_{0} :[0,1]\to \R $ is a given function. Using \cite[Theorem 2.1]{KaK17a}, we are in a position to verify 
assumptions (H1), (H2), (H3) with
\begin{itemize}
\item  state space $X:=C^{2} ([0,1])$,

\item  input set $\Uc(x):=\Big\{(d_{0},d_{1})\in C^{2} (\R _{+} )\times C^{2} (\R _{+})\, :\, d_{0} (0)=x(0)\, ,\, \mathop{\sup }\limits_{t\ge 0} \left|d_{0} (t)\right|<+\infty ,\, d_{1} (0)=x(1) ,\, \mathop{\sup }\limits_{t\ge 0} \left|d_{0} (t)\right|<+\infty \Big\}$,

\item  $G:=(0,1)$, $f(z,w,\xi ):=0$, $(Lx[t])(z):=\frac{\partial \, x}{\partial \, t} (t,z)-a\frac{\partial ^{2} \, x}{\partial \, z^{2} } (t,z)$. 
\end{itemize}

\cite[Corollary 2.5]{KaK17a} ensures the following estimates for all $t\ge 0$, $x_{0} \in X$ and $(d_{0},d_{1}) \in \Uc(x_{0} )$:
\begin{equation} 
\label{GrindEQ__50_} 
\int _{0}^{1}\sin (\pi z)\left|x(t,z)\right|dz \le e^{-\, a\pi ^{2} \, t}\, \int _{0}^{1}\sin (\pi z)\left|x_{0} (z)\right|dz 
+\frac{\, 1}{\pi } \mathop{\max }\limits_{0\le s\le t} \left|d_{0} (s)\right|
+\frac{\, 1}{\pi } \mathop{\max }\limits_{0\le s\le t} \left|d_{1} (s)\right|,
\end{equation} 

\begin{equation} 
\label{GrindEQ__51_} 
\left\| x[t]\right\| _{2} \le \sqrt{\frac{e^{-a\pi ^{2} \, t}}{2-e^{-a\pi ^{2} \, t}} } \left\| x_{0} \right\| _{2} 
+\frac{1}{\sqrt{3} } \mathop{\max }\limits_{0\le s\le t} \left|d_{0} (s)\right|                        
+\frac{1}{\sqrt{3} } \mathop{\max }\limits_{0\le s\le t} \left|d_{1} (s)\right|,
\end{equation} 

\begin{equation} 
\label{GrindEQ__52_}
\begin{split}
\mathop{\max }\limits_{0\le z\le 1} & \frac{\sin (\theta +\varphi )\left|x(t,z)\right|}{\sin \left(\theta +z\varphi \right)} \\
&\quad\le \max \left(e^{-\sigma \, t}\, \mathop{\max }\limits_{0\le z\le 1} \frac{\sin (\theta +\varphi )\left|x_{0} (z)\right|}{\sin \left(\theta +z\varphi \right)}\, ,\, \frac{\sin \left(\theta +\varphi \right)}{\sin \left(\theta \right)} \mathop{\max }\limits_{0\le s\le t} \left|d_{0} (s)\right|,\mathop{\max }\limits_{0\le s\le t} \left|d_{1} (s)\right|\right)\, ,
\end{split}
\end{equation}
for all $\sigma \in (0, a\pi ^{2} )$, $\theta \in (0,\pi -\varphi )$ with $\varphi :=\sqrt{\frac{\sigma }{a} } $.

It is clear that estimates \eqref{GrindEQ__50_}, \eqref{GrindEQ__51_}, \eqref{GrindEQ__52_} are ISS estimates with respect to the boundary disturbances $d_{0},d_{1} :\R _{+} \to \R $ expressed in $L^{2} ,L^{\infty}$ and weighted $L^1$ norms, respectively.

Here, we prove the following property: for every $p\in (2,+\infty )$ there exist constants $M_{p} ,\sigma _{p} ,\gamma _{p} >0$ such that for every $x_{0} \in X$, $(d_{0},d_{1}) \in \Uc(x_{0})$ and $t\ge 0$ the following estimate holds:
\begin{equation} 
\label{GrindEQ__53_} 
\left\| x[t]\right\| _{p} \le M_{p} e^{-\sigma _{p} \, t}\, \left\| x_{0} \right\| _{p} 
+\gamma _{p} \mathop{\max }\limits_{0\le s\le t} \left|d_{0} (s)\right|
+\gamma _{p} \mathop{\max }\limits_{0\le s\le t} \left|d_{1} (s)\right|.
\end{equation} 

First we show that $0\in X$ is exponentially stable in the norm of $L^{p} \left(0,1\right)$ for the system \eqref{GrindEQ__47_}, \eqref{GrindEQ__48_} with $p\in (2,+\infty )$, $d_{0} (t)\equiv 0$, $d_{1} (t)\equiv 0$. This follows from the consideration of the Lyapunov functional 
\[
V_{p} (x)=\int _{0}^{1}\left|x(z)\right|^{p} dz,\quad p\in (2,+\infty).
\]
We get for every $p\in (2,+\infty )$, $x_{0} \in X$ and $t>0$ for the solution $x[t]$ of \eqref{GrindEQ__47_}, \eqref{GrindEQ__48_} with $d_{0} (t)\equiv 0$, $d_{1} (t)\equiv 0$:
\begin{eqnarray*}
\frac{d}{dt} V_{p} (x[t])&=&p\int _{0}^{1}\left|x(t,z)\right|^{p-1} \frac{\partial \, x}{\partial \, t} (t,z)dz =ap\int _{0}^{1}\left|x(t,z)\right|^{p-1} \frac{\partial ^{2} \, x}{\partial \, z^{2} } (t,z)dz  \\ 
&=& -ap(p-1)\int _{0}^{1}\left|x(t,z)\right|^{p-2} \left(\frac{\partial \, x}{\partial \, z} (t,z)\right)^{2} dz\\
&=& -ap(p-1)\frac{4}{p^{2} } \int _{0}^{1}\left(\frac{\partial \, }{\partial \, z} \left|x(t,z)\right|^{p/2} \right)^{2} dz  \\ 
&\le& -ap(p-1)\frac{4\pi^2}{p^{2} } \int _{0}^{1}\left|x(t,z)\right|^{p} dz \\
&=&-ap(p-1)\frac{4\pi^2}{p^{2} } V_{p} (x[t]) 
=-a(p-1)\frac{4\pi^2}{p} V_{p} (x[t]).
\end{eqnarray*}

In the above derivation we used the Wirtinger's inequality $\left\| f'\right\| _{2}^{2} \ge \pi^2\left\| f\right\| _{2}^{2}$, which holds for all $f\in C^{1} ([0,1])$ with $f(0)=f(1)=0$. The above inequality implies the estimate 
\[
V_{p} (x[t])\le e^{-a(p-1)\frac{4\pi^2}{p} t}V_{p} (x_{0} ),\quad t\ge 0,
\]
which in turn shows the exponential stability
\[
\left\| x[t]\right\| _{p} \le e^{-a(p-1)\frac{4\pi^2}{p^2} t}\, \left\| x_{0} \right\| _{p}, \quad t\ge 0.
\]
%
%
%
Based on the above exponential stability estimate and using \cite[Proposition 3]{DaM13} in a similar manner as in the proof of Corollary~\ref{cor:ISS_Linear_Heat_Equation}, we may show that statement (iii) of Theorem~\ref{thm:exp-ISS_Parabolic_Systems_Characterization} holds, and exp-ISS with a linear gain function of 
\eqref{GrindEQ__47_}, \eqref{GrindEQ__48_}, \eqref{GrindEQ__49_} follows from Theorem~\ref{thm:exp-ISS_Parabolic_Systems_Characterization}.
In other words, for every $p\in (2,+\infty )$ there exist constants $M_{p} ,\sigma _{p} ,\gamma _{p} >0$ such that for every $x_{0} \in X$, $(d_{0},d_{1}) \in \Uc(x_{0})$ and $t\ge 0$ the following estimate holds:

\begin{equation} 
\label{eq:Tmp_ISS_estimate} 
\left\| x[t]\right\| _{p} \le M_{p} e^{-\sigma _{p} \, t}\, \left\| x_{0} \right\| _{p} 
+\gamma _{p} \mathop{\max }\limits_{s\ge 0} \left|d_{0} (s)\right|
+\gamma _{p} \mathop{\max }\limits_{s\ge 0} \left|d_{1} (s)\right|.
\end{equation} 
Now for any $\eps>0$ one can find $(\tilde d_{0},\tilde d_{1}) \in \Uc(x_{0})$  so that 
$\tilde d_i(s) = d_i(s)$ for $i=1,2$ and $s\in[0,t]$ and $\sup_{s\geq 0}|\tilde d_i(s)| \leq \sup_{s\in[0,t]}|d_i(s)| + \eps$.

A possible choice for $\tilde d_i$, $i=1,2$ could be
\[
\tilde{d}_i(s):=\begin{cases}
d_i(s) & \text{ if } s\in[0,t], \\ 
\text{sufficiently smooth} & \text{ if } s\in[t,t+\delta], \\ 
d_i(t+\delta) & \text{ if } s\ge t,
\end{cases}
\]
where $\delta=\delta(\eps)$ is chosen small enough, and $d_i$ is approximated on $[t,t+\delta]$ by a sufficiently smooth function so that $\tilde d_0, \tilde d_1$ satisfy above properties.
Now applying \eqref{eq:Tmp_ISS_estimate} to the same $x_0$, $t$ and disturbances $\tilde d_0,\tilde d_1$, and recalling causality property ($\Sigma$\ref{axiom:Causality}), which shows that $\phi(s,x_0,(d_0,d_1)) = \phi(s,x_0,(\tilde d_0,\tilde d_1))$ for $s\in [0,t]$, we obtain
\begin{equation} 
\label{eq:Tmp_ISS_estimate-2} 
\left\| x[t]\right\| _{p} \le M_{p} e^{-\sigma _{p} \, t}\, \left\| x_{0} \right\| _{p} 
+\gamma _{p} (\mathop{\max }\limits_{s\in[0,t]} \left|d_{0} (s)\right| + \eps)
+\gamma _{p} (\mathop{\max }\limits_{s\in[0,t]} \left|d_{1} (s)\right| + \eps).
\end{equation} 
Since $\eps>0$ has been chosen arbitrarily and the rhs of \eqref{eq:Tmp_ISS_estimate-2} depends continuously on $\eps$, we can take the limit $\eps\to+0$, which proves \eqref{GrindEQ__53_}.
 
The above result is important for control purposes. The authors of \cite{SmK04,KrS08} introduced the exponentially stabilizing feedback design for of parabolic PDEs of the form 
\begin{eqnarray}
\frac{\partial \, y}{\partial \, t} (t,z)=a\frac{\partial ^{2} \, y}{\partial \, z^{2} } (t,z)+ k \, y(t,z),\quad (t,z)\in (0,+\infty )\times (0,1),
\label{GrindEQ__54_}
\end{eqnarray}
where $a>0$, $k \in \R $ are constants, subject to the boundary conditions 
\begin{eqnarray}
y(t,0)-u(t)=y(t,1)=0,\quad t\ge 0,
\label{GrindEQ__55_}    
\end{eqnarray}
where $u(t)\in \R $ is the control input, by means of a boundary feedback stabilizer of the form
\begin{eqnarray}
u(t)=-\int _{0}^{1}k(0,s)y(t,s)ds,\quad t\ge 0,
\label{GrindEQ__56_}
\end{eqnarray}
where $k\in C^{2} \left([0,1]^{2} \right)$ is an appropriate function. The function $k\in C^{2} \left([0,1]^{2} \right)$ is obtained as the Volterra kernel of a Volterra integral transformation 
\begin{eqnarray}
x(t,z)=y(t,z)+\int _{z}^{1}k(z,s)y(t,s)ds,\quad (t,z)\in \R _{+} \times [0,1],
\label{GrindEQ__57_}
\end{eqnarray}
which transforms the PDE problem \eqref{GrindEQ__54_}, \eqref{GrindEQ__55_}, \eqref{GrindEQ__56_} to the problem \eqref{GrindEQ__47_} subject to the boundary conditions 
\begin{eqnarray}
x(t,0)=x(t,1)=0,\quad t\ge 0.
\label{GrindEQ__58_}
\end{eqnarray}

The solution of the original problem can be found by the inverse Volterra integral transformation 
\begin{eqnarray}
y(t,z)=x(t,z)+\int _{z}^{1}l(z,s)x(t,s)ds,\quad (t,z)\in \R _{+} \times [0,1],
\label{GrindEQ__59_}
\end{eqnarray}
where $l\in C^{2} \left([0,1]^{2} \right)$ is an appropriate kernel. The existence of the kernels $k\in C^{2} \left([0,1]^{2} \right)$ and $l\in C^{2} \left([0,1]^{2} \right)$ is guaranteed by the main results in \cite{KrS08}. It should be remarked that in \cite{KrS08} the control input is applied at $z=1$ instead of $z=0$, but the transformation of the spatial variable $z\mapsto 1-z$ allows the statement of the results in the above form (with the control action applied at $z=0$). 

When control actuator errors are present, i.e., when the applied control action is of the form
\begin{eqnarray}
u(t)=d(t)-\int _{0}^{1}k(0,s)y(t,s)ds,\quad t\ge 0,
\label{GrindEQ__60_}
\end{eqnarray}
where \textit{$d\in C^{2} (\R _{+} )$}, then the transformed solution $x(t,z)$ satisfies \eqref{GrindEQ__47_} subject to the boundary conditions 
\begin{eqnarray}
x(t,0)-d_{0} (t)=x(t,1)=0,\quad t\ge 0.
\label{GrindEQ__61_}
\end{eqnarray}
 
Using the transformations \eqref{GrindEQ__57_}, \eqref{GrindEQ__59_}, we obtain the existence of constants $K_{2} >K_{1} >0$ such that the following inequality holds for all $t\ge 0$:
\begin{equation} 
\label{GrindEQ__62_} 
K_{1} \left\| x[t]\right\| _{p} \le \left\| y[t]\right\| _{p} \le K_{2} \left\| x[t]\right\| _{p}.
\end{equation} 

Therefore, using \eqref{GrindEQ__53_} and \eqref{GrindEQ__62_}, we can guarantee that for every $p\in (2,+\infty )$, $x_{0} \in X$: $x_0(1)=0$, $(d_{0},0) \in \Uc(x_{0} )$ and $t\ge 0$ the following estimate holds for the solution of the closed-loop system \eqref{GrindEQ__47_}, \eqref{GrindEQ__61_}:
\begin{equation} 
\label{GrindEQ__63_} 
\left\| y[t]\right\| _{p} \le \frac{K_{2} }{K_{1} } M_{p} e^{-\sigma _{p} \, t}\, \left\| y_{0} \right\| _{p} +K_{2} \gamma _{p} \mathop{\max }\limits_{0\le s\le t} \left|d_{0} (s)\right|.
\end{equation} 
The ISS estimate \eqref{GrindEQ__63_} implies robustness with respect to actuator errors in the norm of $L^{p} \left(0,1\right)$ for $p\in (2,+\infty )$. Similar estimates can be obtained using \eqref{GrindEQ__50_}, \eqref{GrindEQ__51_} and \eqref{GrindEQ__52_} in order to obtain ISS estimates in $L^{2} ,L^{\infty }$ and weighted $L^1$ norms, respectively.

\section{Conclusions}

We presented a new technique for analyzing ISS of linear and nonlinear parabolic equations with boundary inputs.
We prove that parabolic equations with Dirichlet boundary inputs are monotone control systems and use this fact to transform the parabolic system with boundary disturbances into a related system with distributed constant disturbances. We show that ISS of the original equation is equivalent to ISS of the transformed system.
Analysis of ISS of the transformed system is much easier to perform, for example, by means of Lyapunov methods.
We apply our methods to prove that an unstable heat equation with additive actuator disturbances can be ISS stabilized by means of PDE backstepping method. 

Although in this paper we concentrate on parabolic scalar equations and study properties of classical solutions of such equations, the scheme which we have developed here can be useful for other classes of monotone control systems: monotone parabolic systems, ordinary differential equations, ODE-heat cascades, some classes of time-delay systems. We expect that a big part of our analysis can be transferred to study properties of mild solutions of parabolic systems. Finally, some of our results can be used to study ISS of monotone systems of a general nature.




%

\end{document}